\documentclass[11pt,a4paper]{amsart}

\usepackage{amsmath,amssymb}
\usepackage{bm}
\usepackage{graphicx}
\usepackage{ascmac}
\usepackage{amsthm}
\usepackage{tikz}
\usepackage{wrapfig}
 
\theoremstyle{definition}
\newtheorem{thm}{Theorem}[section]

\newtheorem{prop}[thm]{Proposition}
\newtheorem{lmm}[thm]{Lemma}

\newcommand{\pran}{\mathbb{P}^{1, an}}

\newcommand{\Sh}{\mathcal{O}}
\newcommand{\mul}{\mathop{m}}

\title{Ramification loci of non-archimedean cubic rational functions}
\author{Reimi Irokawa}
\date{\today}
\address{Graduate School of Science, Tokyo Insitute of Technology}
\email{irokawa.r.aa@m.titech.ac.jp}
\subjclass[2010]{Primary: 14H0, Secondary: 37P50, 14G22, 26E30}
\keywords{non-archimedean field; ramification locus; Berkovich geometry}
\begin{document}
\maketitle
\begin{abstract}
For a cubic rational function with coefficients in a non-archimedean field $K$ whose residue characteristic is $0$ or greater than $3$, there are $2$ possibilities for the shape of its Berkovich ramification locus, considered as an endomorphism of the Berkovich projective line: one is the connected hull of all the critical points, and the other is consisting of $2$ disjoint segments. In this paper, we list up all the possible forms of cubic rational functions and calculate their ramification loci.
\end{abstract}

\section{Introduction}\label{intro}
\subsection{Main results}\label{Main Results}
Let $K$ be an algebraically closed field with complete non-archimedean and non-trivial valuation, We assume that the residue characteristic of $K$ is $0$ or greater than $3$. $\Sh_K$ be the valuation ring, and $\phi$ be a cubic rational function with coefficients in $K$. Rational functions can be considered as endomorphisms of the Berkovich projective line $\pran$ (for definition, see \cite{B1}). The Berkovich ramification locus, or simply the ramification locus of $\phi$ is defined to be the following set
 \begin{align*}
     \mathcal{R}_{\phi}=\{x\in\pran|\textstyle{\mathop{m}_{\phi}}(x)>1\},
 \end{align*}
 where the symbol $\mul_{\phi}(x)$ is the multiplicity of $\phi$ at $x$, i.e., the degree of the field extention $[\kappa(x):\kappa(\phi(x))]$, where the field $\kappa(x)$ is the complete residue field at $x$ (for details and another description of $\mul_\phi$, see \cite{BR}). The ramification locus is a closed subset of $\pran$.
 
The aim of this paper is to give a complete description of the shape of the ramification locus of any cubic rational function. Rational functions $\phi$ and $\psi$ are \textit{conjugate} if there exists M\"obius transformations $\tau$ and $\sigma$ such that $\phi=\tau\circ\psi\circ\sigma$. Since automorphisms do not change the shape of ramification loci, we will describe it for the following representative of each conjugate class.

First, if there exists a critical point of $\phi$ whose multiplicity is $3$, then the rational function $\phi$ is conjugate to a polynomial. This can be done by taking $\tau$ and $\sigma$ so that the critical point with multiplicity $3$ of $\tau\circ\phi\circ\sigma$ is $\infty$. Otherwise, taking a suitable $\tau$ and $\sigma$, we may assume the following conditions;

 \begin{enumerate}
     \item $0$ and $1$ are fixed critical points, 
     \item $\infty$ is fixed but \underline{not} critical, and
     \item the other $2$ critical points are distinct.
 \end{enumerate}
  The cubic rational function $\phi$ with the above conditions can be put
\begin{align}
    \phi(z) = \frac{a_3 z^3+a_2 z^2}{b_2 z^2+b_1 z+b_0}=\frac{(1-\alpha)(1-\beta)z^2(z-\gamma)}{(1-\gamma)(z-\alpha)(z-\beta)}, \tag{$\diamondsuit$} \label{phi}
\end{align}
where $a_2, a_3 ,b_0, b_1, b_2\in \Sh_K$ and $\alpha,\beta,\gamma\in K$. Set $f(z)=a_3z^3+a_2z^2$ and $g(z)=b_2z^2+b_1z+b_0$. To satisfy the above $3$ conditions, we further assume several conditions on them; for details, see the next subsection. Throughout this paper, we consider polynomials or rational functions of this form to calculate the ramification locus. Our result is briefly stated as follows:

\begin{thm}
The ramification locus of a cubic rational function $\phi$ is connected if and only if $\phi$ is conjugate to a polynomial or a rational function of the above forms with the following conditions:
    \begin{itemize}
         \item $|a_2|<1$ and $|a_3|=|b_0|=1$,
         \item $|b_2|<1$ and $|a_3|=|a_2|=|b_0|=|g(\gamma)|=1$, 
         \item $|a_3|=|a_2|=|b_2|=|b_0|=|g(\gamma)|=1$,
         \item $|a_3|,|a_2|<1$, $|a_3|\geq|a_2|$ and $|\gamma-1|=1$,
         \item $|a_3|,|b_1|,|b_0|<1$, $|a_2|=1$, $|a_3|=|b_0|$ and $|a_3|\geq|b_1|$,
         \item $|a_3|=|a_2|=|b_2|=|b_0|=1$,  $|g(\gamma)|<1$ and $|\gamma-1|\leq|\beta-(1/2)|<1$.
    \end{itemize}

The ramification locus of $\phi$ consists of $2$ disjoint segments if and only if $\phi$ is conjugate to a rational function of the above forms with the following conditions:
\begin{itemize}
         \item $|a_3|<1$ and $|a_2|=|b_0|=1$,
         \item $|b_0|<1$ and $|a_3|=|b_2|=|b_1|=1$,
         \item $|b_0|,|b_2|<1$ and $|a_3|=|a_2|=|b_1|=1$,
         \item $|b_2|<1$, $|a_3|=|a_2|=|b_0|=1$ and $|g(\gamma)|<1$,
        \item $|a_3|,|a_2|<1$ and $|g(\gamma)|<1$,
        \item $|a_3|,|a_2|<1$ and $|a_2|>|a_3|$,
        \item $|a_3|,|b_1|,|b_0|<1$, $|a_2|=1$ and $|b_1|>|a_3|$,
        \item $|a_3|,|b_1|,|b_0|<1$, $|a_2|=1$, $|a_3|>|b_0|$ and $|a_3|\geq|b_1|$,
        \item $|a_3|=|a_2|=|b_2|=|b_0|=1$, $|g(\gamma)|<1$ and $|\gamma-1|\neq1$, 
        \item $|a_3|=|a_2|=|b_2|=|b_0|=1$, $|g(\gamma)|<1$ and $|\beta-(1/2)|\neq1$,
        \item $|a_3|=|a_2|=|b_2|=|b_0|=1$, $|g(\gamma)|<1$ and $1>|\gamma-1|>|\beta-(1/2)|$.
\end{itemize}
\end{thm}

In the later sections, we see more detailed information about components of the ramification loci.

This research comes from the works of Faber \cite{F1} and \cite{F2}. There he studies the shape of the ramification locus of rational functions over the Berkovich projective line. When the ramification is tame, then the ramification locus of a rational function is a subgraph of the connected hull of all the critical points by  \cite[Corollary 6.6]{F1}. Also, in general, 
a component containing a point (not necessarily classical) of multiplicity $m$ has at least $2m-2$ critical points counted with multiplicity (see \cite[Theorem A]{F1}). Therefore, in degree $3$ case, since the ramification is always tame when the residue characteristic of $K$ is $0$ or greater than $3$, there are at most $2$ connected components in the ramification locus since the cubic rational functions have $4$ critical points counted with multiplicity. Thus, this is the first non-trivial case; ramification loci of polynomials, rational functions of good reduction, and quadratic rational functions are always connected since in the former two cases the function has a point with multiplicity $d$, and in the last case the function has only $2$ critical points.

\subsection{General strategy}\label{General Strategy}

Let $k$ be the residue field of $K$. For any $a\in\Sh_K$, the symbol $\overline{a}\in k$ denotes its reduction. In the same way, the reduction of any rational function $\psi\in \Sh_K[z]$ is denoted by $\overline{\psi}$. By definition, $\overline{a}=0$ is equivalent to the condition $|a|<1$. For a fixed coordinate of $\pran$, denote its Gauss point by $\zeta_{0,1}$.

For a rational function as in (\ref{phi}) in Section \ref{Main Results}, we may assume that
\begin{itemize}
    \item $a_3\neq0$,
    \item $b_2\neq0$,
    \item at least one of $a_3,a_2,b_2,b_1$ or $b_1$ is invertible, and
    \item polynomials $f$ and $g$ have no common root i.e.\ $\alpha\neq0,$ $\gamma$ and $\beta\neq0$, $\gamma$.
\end{itemize}
Also we have the following equations about the coefficients:
\begin{gather*}
    \phi(1)=1, \text{and}\\
    \text{Wr}_{\phi}(1)=0,
\end{gather*}
where the Wr$_{\phi}(z)$ is the Wronskian of $\phi$:
\begin{align*}
    \text{Wr}_{\phi}(z)=(3a_2z^2+2a_2z)(b_2z^2+b_1z+b_0)-(2b_2z+b_1)(a_3z^3+a_2z^2),
\end{align*}
The first condition is equivalet to
\begin{align}
    a_3+a_2=b_2+b_1+b_0 \tag{$\heartsuit$} \label{phi1}.
\end{align}
The second condition is equivalent to $(3a_3+2a_2)(b_2+b_1+b_0)-(a_3+a_2)(2b_2+b_1)=0$, i.e.,
\begin{align}
    3a_3+2a_2-2b_2-b_1=0 \tag{$\clubsuit$} \label{wr1}
\end{align}
under the condition (\ref{phi1}). We can then list up all the possible cases for the coefficients under these conditions.
\begin{description}
\item[(1-1)] When $\overline{a}_3=\overline{a}_2=0$, we have $\phi(\zeta_{0,1})\neq\zeta_{0,1}$, which is treated in Section \ref{deg0}. This situation is divided into the following $3$ cases:
\begin{description}
    \item[(1-1-1-1)] $|\gamma|\leq1$ i.e.\ $|a_2|\leq|a_3|$, and $\overline{g(\gamma)}=0$,
    \item[(1-1-1-2)] $|\gamma|\leq1$ and $\overline{g(\gamma)}\neq0$,
    \item[(1-1-2)] $|\gamma|>1$ i.e.\ $|a_2|>|a_3|$.
\end{description}
\item[(1-2-1-1)] When $\overline{a}_3=\overline{b}_0=\overline{b}_1=0$ and $\overline{a}_2\neq0$, we have $\phi(\zeta_{0,1})\neq\zeta_{0,1}$, which is treated in Section \ref{deg0}. This situation is divided into the following $3$ cases:
\begin{description}
    \item[(1-2-1-1-1)] $|b_1|>|a_3|$,
    \item[(1-2-1-1-2)] $|b_1|\leq|a_3|$ and $|a_3|>|b_0|$,
    \item[(1-2-1-1-3)] $|a_3|=|b_0|=|b_1|$.
\end{description}
Any other condition on $a_3$, $b_1$ and $b_0$ is impossible by (\ref{phi1}) and (\ref{wr1}).
\item[(1-2-1-2)] When $\overline{a}_3=\overline{b}_0=0$, $\overline{a}_2\neq0$ and $\overline{b}_1\neq0$, the degree of $\overline{\phi}$ is $1$. This case is treated in Section \ref{deg1}.
\item[(1-2-2)] When $\overline{a}_3=0$, $\overline{a}_2\neq0$ and $\overline{b}_0\neq0$, the degree of $\overline{\phi}$ is $2$, which is treated in Section \ref{deg2}.
\item[(2-1-1)] when $\overline{a}_3\neq0$ and $\overline{a}_2=\overline{b}_0=0$, the degree of $\overline{\phi}$ is 1. It is treated in Section \ref{deg1}.
\item[(2-1-2)] When $\overline{a}_3\neq0$, $\overline{a}_2=0$ and $\overline{b}_0\neq0$, the function $\phi$ has good reduction i.e., the ramification locus is connected.
\item[(2-2)] When $\overline{a}_3\neq0$ and $\overline{a}_2\neq0$, the degree of $\overline{\phi}$ depends on whether$\overline{g(\gamma)}$ is zero or not, and whether $\overline{b}_0$ is zero or not.
\begin{description}
\item[(2-2-1-1)] When $\overline{b}_0=\overline{b}_2=0$, the degree of $\phi$ is $2$. This case is treated in Section \ref{deg2}.
\item[(2-2-1-2-1)] When $\overline{b}_0=\overline{g(\gamma)}=0$ and $\overline{b}_2\neq0$, the degree of $\overline{\phi}$ is 1. It is treated in Section \ref{deg1}.
\item[(2-2-1-2-2)] When $\overline{b}_0=0$, $\overline{b}_2\neq0$ and $\overline{g(\gamma)}\neq0$, the degree of $\overline{\phi}$ is $2$. It is treated in Section \ref{deg2}.
\item[(2-2-2-1-1)] When $\overline{b}_2=\overline{g(\gamma)}=0$ and $\overline{b}_0\neq0$, the degree of $\overline{\phi}$ is $1$, which is treated in Section \ref{deg2}.
\item[(2-2-2-1-2)] When $\overline{b}_2=0$ and $\overline{b}_0$, $\overline{g(\gamma)}\neq0$, the degree of $\overline{\phi}$ is $3$ i.e., $\phi$ has good reduction and the ramification locus is connected;
\item[(2-2-2-2-1)] When $\overline{b}_0\neq0$, $\overline{b}_2\neq0$ and $\overline{g(\gamma)}=0$, the degree of $\overline{\phi}$ is $2$. Later, we will devide this case into further two cases as follows:
\begin{description}
\item[(2-2-2-2-1-1)] $\overline{\alpha}=\overline{\gamma}=1$ and $\overline{\beta}=1/2$;
\item[(2-2-2-2-1-2)] otherwise.
\end{description}
The former case is treated in Section \ref{deg2-2}, and the latter case is treated in Section \ref{deg2};
\item[(2-2-2-2-2)] When $\overline{b}_0$, $\overline{b}_2$, $\overline{g(\gamma)}\neq0$, the function $\phi$ has good reduction i.e., the ramification locus is connected.
\end{description}
\end{description}

The numbering is due to Figure \ref{a3eq0} and Figure \ref{a3neq0}.

\begin{figure}
\centering
\[
	\begin{tikzpicture}[scale = 1.0, auto] 
		\node[draw, rectangle] (a3eq0) at (0, 0) {$\overline{a}_3=0$};
		\node[draw, circle, inner sep = 0pt, minimum size =  1pt, fill = black] (a2) at (0,-0.5) {};
		\node[draw, circle, inner sep = 0pt, minimum size = 1pt, fill = black] (a2l) at (-5,-0.5) {};
		\node[draw, circle, inner sep = 0pt, minimum size = 1pt, fill = black] (a2r) at (5,-0.5) {};
		\node[draw, rectangle] (a2=0) at (-5,-1) {$\overline{a}_2=0$};
		\node[draw, rectangle] (a2neq0) at (5, -1) {$\overline{a}_2\neq0$};		
		\draw (a3eq0) to  (a2);
		\draw (a2) to (a2l);
		\draw (a2) to (a2r);
		\draw (a2l) to (a2=0);
		\draw (a2r) to (a2neq0);
		
		\node[draw, circle, inner sep = 0pt, minimum size = 1pt, fill = black] (gamma) at (-5,-1.5) {};
		\node[draw, circle, inner sep = 0pt, minimum size = 1pt, fill = black] (gammal) at (-6.5,-1.5) {};
		\node[draw, circle, inner sep = 0pt, minimum size = 1pt, fill = black] (gammar) at (-3.5,-1.5) {};
		\node[draw, rectangle] (gammaleq1) at (-6.5, -2) {$|\gamma|\leq1$};
		\node[draw, rectangle] (gamma>1) at (-3.5, -2) {$|\gamma|>1$};
		\draw (a2=0) to  (gamma);
		\draw (gamma) to (gammal);
		\draw (gamma) to (gammar);
		\draw (gammal) to (gammaleq1);
		\draw (gammar) to (gamma>1);

		\node[draw, rectangle, rounded corners] (case1-1-2) at (-3.5, -3) {case(1-1-2)};
		\node[draw, rectangle] (ggamma=0) at (-8,-4) {$\overline{g(\gamma)}=0$};
		\node[draw, rectangle] (ggammaneq0) at (-5, -4) {$\overline{g(\gamma)}\neq0$};
		\node[draw, circle, inner sep = 0pt, minimum size = 1pt, fill = black] (ggamma) at (-6.5,-3) {};
		\node[draw, circle, inner sep = 0pt, minimum size = 1pt, fill = black] (ggammal) at (-8,-3) {};
		\node[draw, circle, inner sep = 0pt, minimum size = 1pt, fill = black] (ggammar) at (-5,-3) {};
		\draw (gamma>1) to  (case1-1-2);
		\draw (gammaleq1) to (ggamma);
		\draw (ggamma) to (ggammal);
		\draw (ggamma) to (ggammar);
		\draw (ggammal) to (ggamma=0);
		\draw (ggammar) to (ggammaneq0);

		\node[draw, rectangle, rounded corners] (case1-1-1-1) at (-8, -5) {case(1-1-1-1)};
		\node[draw, rectangle, rounded corners] (case1-1-1-2) at (-5, -5) {case(1-1-1-2)};
		\draw (ggamma=0) to (case1-1-1-1);
		\draw (ggammaneq0) to (case1-1-1-2);
		
		\node[draw, rectangle] (b0=0) at (3.5,-2) {$\overline{b}_0=0$};
		\node[draw, rectangle] (b0neq0) at (6.5, -2) {$\overline{b}_0\neq0$};
		\node[draw, circle, inner sep = 0pt, minimum size =  1pt, fill = black] (b0) at (5,-1.5) {};
		\node[draw, circle, inner sep = 0pt, minimum size = 1pt, fill = black] (b0l) at (3.5,-1.5) {};
		\node[draw, circle, inner sep = 0pt, minimum size = 1pt, fill = black] (b0r) at (6.5,-1.5) {};
		\draw (a2neq0) to  (b0);
		\draw (b0) to (b0l);
		\draw (b0) to (b0r);
		\draw (b0l) to (b0=0);
		\draw (b0r) to (b0neq0);
		
		\node[draw, rectangle, rounded corners] (case1-2-2) at (6.5,-3) {case(1-2-2)};
		\draw (b0neq0) to (case1-2-2);

		\node[draw, rectangle] (b1=0) at (-0.5,-3) {$\overline{b}_1=0$};
		\node[draw, rectangle] (b1neq0) at (4.5, -3) {$\overline{b}_1\neq0$};
		\node[draw, circle, inner sep = 0pt, minimum size =  1pt, fill = black] (b1) at (3.5,-2.5) {};
		\node[draw, circle, inner sep = 0pt, minimum size = 1pt, fill = black] (b1l) at (-0.5,-2.5) {};
		\node[draw, circle, inner sep = 0pt, minimum size = 1pt, fill = black] (b1r) at (4.5,-2.5) {};
		\draw (b0=0) to  (b1);
		\draw (b1) to (b1l);
		\draw (b1) to (b1r);
		\draw (b1l) to (b1=0);
		\draw (b1r) to (b1neq0);

		\node[draw, rectangle, rounded corners] (case1-2-1-2) at (4.5,-4) {case(1-2-1-2)};
		\draw (b1neq0) to (case1-2-1-2);

		\node[draw, rectangle] (b1>a3) at (-2.5,-5) {$|b_1|>|a_3|$};
		\node[draw, align=right, rectangle] (a3>b0) at (-0.5, -5) {$|b_1|\leq|a_3|$ \\ $|a_3|>|b_0|$};
		\node[draw, rectangle] (a3=b0=b1) at (2, -5) {$|a_3|=|b_0|=|b_1|$};
		
		\node[draw, circle, inner sep = 0pt, minimum size =  1pt, fill = black] (a3b0b1l) at (-2.5,-4) {};
		\node[draw, circle, inner sep = 0pt, minimum size = 1pt, fill = black] (a3b0b1c) at (-0.5,-4) {};
		\node[draw, circle, inner sep = 0pt, minimum size = 1pt, fill = black] (a3b0b1r) at (2,-4) {};

		\draw (b1=0) to  (a3b0b1c);
		\draw (a3b0b1l) to (a3b0b1r);
		\draw (a3b0b1l) to (b1>a3);
		\draw (a3b0b1c) to (a3>b0);
		\draw (a3b0b1r) to (a3=b0=b1);

		\node[draw, rectangle, rounded corners] (case1-2-1-1-1) at (-2.5,-6) {case(1-2-1-1-1)};
		\draw (b1>a3) to (case1-2-1-1-1);

		\node[draw, rectangle, rounded corners] (case1-2-1-1-2) at (-0.5,-6.8) {case(1-2-1-1-2)};
		\draw (a3>b0) to (case1-2-1-1-2);

		\node[draw, rectangle, rounded corners] (case1-2-1-1-3) at (2,-6) {case(1-2-1-1-3)};
		\draw (a3=b0=b1) to (case1-2-1-1-3);
\end{tikzpicture}
\]
\caption{The case $\overline{a}_3=0$}\label{a3eq0}
\end{figure}

\begin{figure}
\centering
\[
	\begin{tikzpicture}[scale = 1.0, auto]
		\node[draw, rectangle] (a3neq0) at (0,0) {$\overline{a}_3\neq0$};
		\node[draw, rectangle] (a2=0) at (-4,-1) {$\overline{a}_2=0$};
		\node[draw, rectangle] (a2neq0) at (0,-1) {$\overline{a}_2\neq0$};
		\node[draw, circle, inner sep = 0pt, minimum size = 1pt, fill = black] (a2l) at (-4,-0.5) {};
		\node (a2r)[draw, circle, inner sep = 0pt, minimum size = 1pt, fill = black] at (0,-0.5) {};

		\draw (a3neq0) to (a2neq0);
		\draw (a2l) to (a2=0);
		\draw (a2l) to (a2r);

		\node[draw, rectangle] (b0=0) at (-6,-2) {$\overline{b}_0=0$};
		\node[draw, rectangle] (b0neq0) at (-4,-2) {$\overline{b}_0\neq0$};
		\node[draw, circle, inner sep = 0pt, minimum size = 1pt, fill = black] (b0l) at (-6,-1.5) {};
		\node[draw, circle, inner sep = 0pt, minimum size = 1pt, fill = black]  (b0r) at (-4,-1.5) {};

		\draw (a2=0) to (b0neq0);
		\draw (b0l) to (b0=0);
		\draw (b0l) to (b0r);
		
		\node[draw, rectangle, rounded corners] (case2-1-1) at (-6,-3) {case(2-1-1)};
		\draw (b0=0) to (case2-1-1);

		\node[draw, rectangle, rounded corners] (case2-1-2) at (-4,-3) {case(2-1-2)};
		\draw (b0neq0) to (case2-1-2);

		\node[draw, rectangle] (b0=0) at (0,-2) {$\overline{b}_0=0$};
		\node[draw, rectangle] (b0neq0) at (3,-2) {$\overline{b}_0\neq0$};
		\node[draw, circle, inner sep = 0pt, minimum size = 1pt, fill = black] (b0l) at (0,-1.5) {};
		\node (b0r)[draw, circle, inner sep = 0pt, minimum size = 1pt, fill = black] at (3,-1.5) {};

		\draw (a2neq0) to (b0=0);
		\draw (b0r) to (b0neq0);
		\draw (b0l) to (b0r);

		\node[draw, rectangle] (b2=0) at (-2,-3) {$\overline{b}_2=0$};
		\node[draw, rectangle] (b2neq0) at (0,-3) {$\overline{b}_2\neq0$};
		\node[draw, circle, inner sep = 0pt, minimum size = 1pt, fill = black] (b2l) at (-2,-2.5) {};
		\node[draw, circle, inner sep = 0pt, minimum size = 1pt, fill = black] (b2r) at (0,-2.5) {};

		\draw (b0=0) to (b2neq0);
		\draw (b2l) to (b2=0);
		\draw (b2l) to (b2r);

		\node[draw, rectangle, rounded corners] (case2-2-1-1) at (-2,-4) {case(2-2-1-1)};
		\draw (case2-2-1-1) to (b2=0);

		\node[draw, rectangle] (ggamma=0) at (-1,-5) {$\overline{g(\gamma)}=0$};
		\node[draw, rectangle] (ggammaneq0) at (1,-5) {$\overline{g(\gamma)}\neq0$};
		\node[draw, circle, inner sep = 0pt, minimum size = 1pt, fill = black] (ggammal) at (-1,-4.5) {};
		\node[draw, circle, inner sep = 0pt, minimum size = 1pt, fill = black] (ggammar) at (1,-4.5) {};
		\node[draw, circle, inner sep = 0pt, minimum size = 1pt, fill = black] (ggammac) at (0,-4.5) {};

		\draw (b2neq0) to (ggammac);
		\draw (ggammal) to (ggammar);
		\draw (ggammal) to (ggamma=0);
		\draw (ggammar) to (ggammaneq0);

		\node[align=center, draw, rectangle, rounded corners] (case2-2-1-2-1) at (-1,-6) {case \\(2-2-1-2-1)};
		\draw (case2-2-1-2-1) to (ggamma=0);
		\node[align=center, draw, rectangle, rounded corners] (case2-2-1-2-2) at (1,-6) {case \\(2-2-1-2-2)};
		\draw (case2-2-1-2-2) to (ggammaneq0);
		
		\node[draw, rectangle] (b2=0) at (3,-3) {$\overline{b}_2=0$};
		\node[draw, rectangle] (b2neq0) at (7,-3) {$\overline{b}_2\neq0$};
		\node[draw, circle, inner sep = 0pt, minimum size = 1pt, fill = black] (b2l) at (3,-2.5) {};
		\node[draw, circle, inner sep = 0pt, minimum size = 1pt, fill = black] (b2r) at (7,-2.5) {};

		\draw (b0neq0) to (b2=0);
		\draw (b2r) to (b2neq0);
		\draw (b2l) to (b2r);
		
		\node[draw, rectangle] (ggamma=0) at (3,-5) {$\overline{g(\gamma)}=0$};
		\node[draw, rectangle] (ggammaneq0) at (5,-5) {$\overline{g(\gamma)}\neq0$};
		\node[draw, circle, inner sep = 0pt, minimum size = 1pt, fill = black] (ggammal) at (3,-4.5) {};
		\node[draw, circle, inner sep = 0pt, minimum size = 1pt, fill = black] (ggammar) at (5,-4.5) {};

		\draw (b2=0) to (ggamma=0);
		\draw (ggammal) to (ggammar);
		\draw (ggammar) to (ggammaneq0);

		\node[draw, align=center, rectangle, rounded corners] (case2-2-2-1-2) at (5,-6) {case \\(2-2-2-1-2)};
		\draw (case2-2-2-1-2) to (ggammaneq0);

		\node[draw, align=center, rectangle, rounded corners] (case2-2-2-1-1) at (3,-6) {case \\(2-2-2-1-1)};
		\draw (case2-2-2-1-1) to (ggamma=0);

		\node[draw, rectangle] (ggamma=0) at (7,-5) {$\overline{g(\gamma)}=0$};
		\node[draw, rectangle] (ggammaneq0) at (9,-5) {$\overline{g(\gamma)}\neq0$};
		\node[draw, circle, inner sep = 0pt, minimum size = 1pt, fill = black] (ggammal) at (7,-4.5) {};
		\node[draw, circle, inner sep = 0pt, minimum size = 1pt, fill = black] (ggammar) at (9,-4.5) {};

		\draw (b2neq0) to (ggamma=0);
		\draw (ggammal) to (ggammar);
		\draw (ggammar) to (ggammaneq0);

		\node[draw, align=center, rectangle, rounded corners] (case2-2-2-2-2) at (9,-6) {case \\(2-2-2-2-2)};
		\draw (case2-2-2-2-2) to (ggammaneq0);

		\node[draw, align=center, rectangle] (alpha=1beta=half) at (5,-7.5) {$\overline{\alpha}=\overline{\gamma}=1$,\\ $\overline{\beta}=1/2$};
		\node[draw, align=center, rectangle] (alpha=1betaneqhalf) at (7,-7.5) {$\overline{\alpha}=\overline{\gamma}=1$, \\$\overline{\beta}\neq1/2$};
		\node[draw, align=center, rectangle] (alphaneq1beta=half) at (9,-7.5) {$\overline{\alpha}=\overline{\gamma}\neq1$,\\ $\overline{\beta}=1/2$};
		\node[draw, circle, inner sep = 0pt, minimum size = 1pt, fill = black] (alphabetal) at (5,-6.8) {};
		\node[draw, circle, inner sep = 0pt, minimum size = 1pt, fill = black] (alphabetac) at (7,-6.8) {};
		\node[draw, circle, inner sep = 0pt, minimum size = 1pt, fill = black] (alphabetar) at (9,-6.8) {};

		\draw (ggamma=0) to (alphabetac);
		\draw (alphabetal) to (alphabetar);
		\draw (alphabetal) to (alpha=1beta=half);
		\draw (alphabetac) to (alpha=1betaneqhalf);
		\draw (alphabetar) to (alphaneq1beta=half);

		\node[align=center, draw, rectangle, rounded corners] (case2-2-2-1-1-1-1) at (5,-9) {case \\(2-2-2-2-1-1)};
		\draw (case2-2-2-1-1-1-1) to (alpha=1beta=half);

		\node[draw, circle, inner sep = 0pt, minimum size = 1pt, fill = black] (alphabetarb) at (7,-8.2) {};
		\node[draw, circle, inner sep = 0pt, minimum size = 1pt, fill = black] (alphabetacb) at (9,-8.2) {};
		\node[draw, circle, inner sep = 0pt, minimum size = 1pt, fill = black] (case2-2-2-1-1-1-2a) at (8,-8.2) {};
		
		\draw (alphabetacb) to (alphaneq1beta=half);
		\draw (alphabetarb) to (alpha=1betaneqhalf);
		\draw (alphabetarb) to (alphabetacb);

		\node[align=center, draw, rectangle, rounded corners] (case2-2-2-1-1-1-2) at (8,-9) {case \\(2-2-2-2-1-2)};
		\draw (case2-2-2-1-1-1-2a) to (case2-2-2-1-1-1-2);

	\end{tikzpicture}
\]
\caption{the case $\overline{a}_3\neq0$} \label{a3neq0}
\end{figure}

Since the Wronskian $\mathop{Wr}_{\phi}(z)$ vanishes at $0$ and $1$, we have
\begin{align*}
    \textstyle{\mathop{Wr}_{\phi}(z)}=z(z-1)\psi(z),
\end{align*}
where
\begin{align}
    \psi(z)=a_3b_2z^2+(2a_3b_1+a_3b_2)z-2a_2b_0. \tag{$\spadesuit$} \label{psi}
\end{align}

In each of the above cases, we compare the zeros of $\overline{\psi}(z)$ and $\mathop{Wr}_{\overline{\phi}}(z)$ to calculate the ramification locus.

\section{Calculation}
\subsection{The case $\phi(\zeta_{0,1})\neq\zeta_{0,1}$} \label{deg0}

Cases in (1-1) requires $\overline{a}_3=\overline{a}_2=0$. It follows from (\ref{wr1}) and (\ref{phi1})that
\begin{align*}
    \overline{b}_2+\overline{b}_1+\overline{b}_0&=0,\text{ and}\\
    2\overline{b}_2+\overline{b}_1&=0,
\end{align*}
from which we have $\overline{b}_0=\overline{b}_2$ and $\overline{b}_1=-2\overline{b}_2$.

In case (1-1-1-1), we have $\overline{g(\gamma)}=\overline{b}_2\overline{\gamma}^2-2\overline{b}_2\overline{\gamma}+\overline{b}_2=\overline{b}_2(\overline{\gamma}-1)^2=0$, i.e., $(\overline{a_2/a_3}=)\overline{\gamma}=-1$.
The polynomial $\psi(z)$ in (\ref{psi}) is
\begin{align*}
    \psi(z)=a_3(b_2z^2+(2b_1+b_2)z-2a_2b_0/a_3),
\end{align*}
The reduction of $\psi/a_3$ is 
\begin{align*}
    \overline{\psi/a_3}(z)&=\overline{b}_2z^2-3\overline{b}_2z-2\overline{b}_0(\overline{a_2/a_3})\\
    &=\overline{b}_2(z^2-3z+2).
\end{align*}
The solutions of $\overline{\psi/a_3}(x)=0$ are $\overline{c}_1=-2$ and $\overline{c}_2=-1.$

On the other hand, since $\phi(\zeta_{0,1})=\zeta_{0,|a_3|}$ in this case, we have
\begin{align*}
    \overline{\phi/a_3}(z)&=\frac{z^2(z-\overline{\gamma})}{\overline{b}_2(z-1)^2} \\
    &=\frac{z^2}{\overline{b}_2(z-1)}.
\end{align*}
The Wronskian is
\begin{align*}
    \textstyle{\mathop{Wr}_{\overline{\phi}}}(z)=\overline{b}_2z(z-2).
\end{align*}
Therefore, the ramification locus has $2$ connected components; one is the segment connecting $0$ and $c_1$ and the other is the one connecting $1$ and $c_2$, as shown in Figure \ref{fig:mphi=0-1}.

The case (1-1-1-2) is when $\phi$ has potentially good reduction; the ramification locus is always connected in this case.

In case (1-1-2), we can calculate $c_1$, $c_2$ and zeros of $\mathop{Wr}_{\overline{\phi}}(z)$ in the similar way as above by replacing $\phi/a_3$ and $\psi/a_2$ by $\phi/a_2$ and $\psi/a_2$ respectively; the zeros of $\overline{\psi/a_2}$ are $\overline{c}_1=\overline{c}_2=\infty$ i.e. they have abosolute value greater than $1$, and the zeros of $\mathop{Wr}_{\overline{\phi}}(z)$ are $0$ and $1$. The ramification locus has hence two connected components; one is the segment connecting $0$ and $1$, and the other is the one connecting $c_1$ and $c_2$.

In cases (1-2-1-1), we have $\overline{a}_2=\overline{b}_2$.

In case (1-2-1-1-1), consider 
\begin{align*}
    \phi'(z)&=\frac{\phi(z)-a_2/b_2}{b_1} \\
    &=\frac{b_2a_3z^3/b_1-a_2z-a_2b_0/b_1}{b_2(b_2z^2+b_1z+b_0)}.
\end{align*}

Since $2\overline{b_0/b_1}=-1$ by (\ref{wr1}), we have
\begin{align*}
    \textstyle{\mathop{Wr}_{\phi'}}(z)&=-\overline{b}_2z^2+2\overline{b}_2z(z-\frac{1}{2}) \\
    &=\overline{b}_2z(z-1).
\end{align*}
By Newton polygon argument, the two zeros of $\psi$ have absolute value greater than $1$. Therefore, the ramification locus has two connected components; one is the segment connecting $0$ and $1$, and the other is the one connecting the remaining two critical points.

In case (1-1-2) and (1-2-1-1-1), the shape of the ramification locus looks like Figure \ref{fig:mphi=0-2}

We can do the similar calculation for the cases (1-2-1-1-2) and (1-2-1-1-3) by replacing $b_1$ by $a_3$.

In case (1-2-1-1-2), $\overline{c}_1=0$ and $\overline{c}_2=-1$. The Wronskian of the reduction of $(\phi-a_2/b_2)/a_3$ is $\overline{b}_2(z+1)(z-1)$. Therefore, the ramification locus has two components; one is the segment connecting $0$ and $c_1$, and the ohter is the one connecting $1$ and $c_2$ i.e.\ as shown in Figure \ref{fig:mphi=1}.

In case (1-2-1-1-3), The reduction of $(\phi-a_2/b_2)/a_3$ has degree $3$ i.e.\ of good reduction. The ramification locus is always connected.

\begin{figure}
    \centering
    \begin{tabular}{ccc}
    
    \begin{minipage}{0.3\hsize}
    	\begin{tikzpicture}[scale=2.0, auto]
		\node[circle, inner sep = 0pt, minimum size =  1pt, fill = black, draw=black] (gauss) at (0,0) {};
		\node[circle, inner sep = 0pt, minimum size =  1pt, fill = black, draw=black] (0) at (0,-1) {};
		\node[circle, inner sep = 0pt, minimum size =  1pt, fill = black, draw=black] (1) at (1,0) {};
		\node[circle, inner sep = 0pt, minimum size =  1pt, fill = black, draw=black] (infty) at (0,1) {};
		\node[circle, inner sep = 0pt, minimum size =  1pt, fill = black, draw=black] (ram2) at (0.5,0) {};
		\node[circle, inner sep = 0pt, minimum size =  1pt, fill = black, draw=black] (c1) at (-0.75,0.75) {};
		\node[circle, inner sep = 0pt, minimum size =  1pt, fill = black, draw=black] (c2) at (0.75,-0.75) {};
		\node[above right] (mgauss) at (0,0) {$\zeta_{0,1}$};
		\node[right] (m1) at (1,0) {$1$};
		\node[below] (m0) at (0,-1) {$0$};
		\node[above] (minfty) at (0,1) {$\infty$};
		\node[above left] (mc1) at (-0.75,0.75) {$c_1$};
		\node[below right] (mc2) at (0.75,-0.75) {$c_2$};

		\draw (infty) to (0);
		\draw (gauss) to (1);
		\draw[very thick] (c1) to (gauss);
		\draw[very thick] (c2) to (ram2);
		\draw[very thick] (1) to (ram2);
		\draw[very thick] (0) to (gauss);
	\end{tikzpicture}
	\caption{} \label{fig:mphi=0-1}
    \end{minipage}&

\begin{minipage}{0.3\hsize}
    	\begin{tikzpicture}[scale=2.0, auto]
		\node[circle, inner sep = 0pt, minimum size =  1pt, fill = black, draw=black] (gauss) at (0,0) {};
		\node[circle, inner sep = 0pt, minimum size =  1pt, fill = black, draw=black] (0) at (0,-1) {};
		\node[circle, inner sep = 0pt, minimum size =  1pt, fill = black, draw=black] (1) at (1,0) {};
		\node[circle, inner sep = 0pt, minimum size =  1pt, fill = black, draw=black] (infty) at (0,1) {};
		\node[circle, inner sep = 0pt, minimum size =  1pt, fill = black, draw=black] (ram1) at (0,0.5) {};
		\node[circle, inner sep = 0pt, minimum size =  1pt, fill = black, draw=black] (ram2) at (0,0) {};
		\node[circle, inner sep = 0pt, minimum size =  1pt, fill = black, draw=black] (c1) at (-0.75,0.75) {};
		\node[circle, inner sep = 0pt, minimum size =  1pt, fill = black, draw=black] (c2) at (0.75,0.75) {};
		\node[left] (mgauss) at (0,0) {$\zeta_{0,1}$};
		\node[right] (m1) at (1,0) {$1$};
		\node[below] (m0) at (0,-1) {$0$};
		\node[above] (minfty) at (0,1) {$\infty$};
		\node[above left] (mc1) at (-0.75,0.75) {$c_1$};
		\node[above right] (mc2) at (0.75,0.75) {$c_2$};

		\draw (infty) to (0);
		\draw (gauss) to (1);
		\draw[very thick] (c1) to (ram1);
		\draw[very thick] (c2) to (ram1);
		\draw[very thick] (0) to (ram2);
		\draw[very thick] (1) to (ram2);
	\end{tikzpicture}
	\caption{} \label{fig:mphi=0-2}
	\end{minipage} &

\begin{minipage}{0.3\hsize}
	\begin{tikzpicture}[scale=2.0, auto]
		\node[circle, inner sep = 0pt, minimum size =  1pt, fill = black, draw=black] (gauss) at (0,0) {};
		\node[circle, inner sep = 0pt, minimum size =  1pt, fill = black, draw=black] (0) at (0,-1) {};
		\node[circle, inner sep = 0pt, minimum size =  1pt, fill = black, draw=black] (1) at (1,0) {};
		\node[circle, inner sep = 0pt, minimum size =  1pt, fill = black, draw=black] (infty) at (0,1) {};
		\node[circle, inner sep = 0pt, minimum size =  1pt, fill = black, draw=black] (ram1) at (0,-0.5) {};
		\node[circle, inner sep = 0pt, minimum size =  1pt, fill = black, draw=black] (ram2) at (0.5,0) {};
		\node[circle, inner sep = 0pt, minimum size =  1pt, fill = black, draw=black] (c1) at (-0.75,-0.75) {};
		\node[circle, inner sep = 0pt, minimum size =  1pt, fill = black, draw=black] (c2) at (0.75,0.75) {};
		\node[left] (mgauss) at (0,0) {$\zeta_{0,1}$};
		\node[right] (m1) at (1,0) {$1$};
		\node[below] (m0) at (0,-1) {$0$};
		\node[above] (minfty) at (0,1) {$\infty$};
		\node[below left] (mc1) at (-0.75,-0.75) {$c_1$};
		\node[above right] (mc2) at (0.75,0.75) {$c_2$};

		\draw (infty) to (0);
		\draw (gauss) to (1);
		\draw[very thick] (c1) to (ram1);
		\draw[very thick] (c2) to (ram2);
		\draw[very thick] (0) to (ram1);
		\draw[very thick] (1) to (ram2);
	\end{tikzpicture}
    \caption{} \label{fig:mphi=1}
\end{minipage} 
\end{tabular}
\end{figure}

\subsection{The case $\mathop{m}_{\phi}(\zeta_{0,1})=1$} \label{deg1}

In this case, the ramification locus must have two connected components and neither of them contains the Gauss point $\zeta_{0,1}$. The remaining critical points $c_1$ and $c_2$ must satisfy that $\overline{c}_1=0$ and $\overline{c}_2=1$. Figure \ref{fig:mphi=1} shows its shape.

\subsection{The case $\mathop{m}_{\phi}(\zeta_{0,1})=2$} \label{deg2}
 In this case, the following three cases are possible:
 \begin{enumerate}
     \item the ramification locus has two connected components, one of which is the segment connecting $0$ and $c_1$ and the other is the one connecting $1$ and $c_2$;
     \item the ramification locus has two connected components, one of which is the segment connecting $0$ and $1$ and the other is the one connecting the two remaining critical points i.e.\ as shown in Figure \ref{fig:mphi=0-2};
     \item the ramification locus is connected.
 \end{enumerate}
 
 If the two remaining critical points are of absolute value greater than $1$ i.e.\ the case (2) above, we must have $|a_3b_2|<1$ and $|a_3b_1+a_3b_0|<1$ in (\ref{psi}). In the list in Section \ref{intro}, it is possible only in case (1-2-2).
 
 By a straightforward calculation similar to that in Section \ref{deg0}, the case (1) happens in any other cases except for the cases (2-2-2-2-1-1) and (2-2-2-2-1-2).
 
 In each of these cases where (1) occurs, the reduction of the zeros of $\psi$ is as follows:
 
 \begin{description}
    \item[(1-2-2)] $|c_1|=|c_2|>1$ i.e.\ $\overline{c}_1=\overline{c}_2=\infty$, and the shape looks like Figure \ref{fig:mphi=1};
    \item[(2-2-1-1)] $\overline{c}_1=0$ and $\overline{c}_2=\infty$, and the shape looks like Figure \ref{fig:mphi=2-1};
    \item[(2-2-1-2-2)] $\overline{c}_1=0$ and $\overline{c}_2=-1-2\overline{b}_1/\overline{b}_2$, and the shape looks like Figure \ref{fig:mphi=2-2}
    \item[(2-2-2-1-1)] $\overline{c}_1=\infty$ and $\overline{c}_2=1$, and the shape looks like Figure \ref{fig:mphi=2-3}.
 \end{description}
 
 Therefore, we calculate the ramification locus when $\overline{a}_3$, $\overline{a}_2$, $\overline{b}_0$, $\overline{b}_2\neq0$ and $\overline{g(\gamma)}=0$. In this case, we have from $\overline{\mathop{Wr}}_{\phi}(1)=0$ 
 that $\overline{\beta}=1/2$ or $\overline{\alpha}=\overline{\gamma}=1$. When either of these two equations fails to hold, we have the case (1). The only non-trivial case is when $\phi$ satisfies the both equations i.e. the case (2-2-2-2-1-1), which is treated in the next subsection.
 
 For the case (2-2-2-2-1-2), the reduction of the remaining critical points are
\begin{description}
    \item[when $\overline{\beta}=1/2$] $\overline{c}_1=\overline{c}_2=\overline{\alpha}$ i.e.\ as shown in Figure \ref{fig:mphi=2-4};
    \item[when $\overline{\alpha}=\overline{\gamma}=1$] $\overline{c}_1=2\overline{\beta}$ and $\overline{c}_2=1$ i.e.\ as shown in Figure \ref{fig:mphi=2-5}.
\end{description}

\begin{figure}[b]
    \centering
    \begin{tabular}{ccc}

\begin{minipage}{0.3\hsize}
     	\begin{tikzpicture}[scale=2.0, auto]
		\node[circle, inner sep = 0pt, minimum size =  1pt, fill = black, draw=black] (gauss) at (0,0) {};
		\node[circle, inner sep = 0pt, minimum size =  1pt, fill = black, draw=black] (0) at (0,-1) {};
		\node[circle, inner sep = 0pt, minimum size =  1pt, fill = black, draw=black] (1) at (1,0) {};
		\node[circle, inner sep = 0pt, minimum size =  1pt, fill = black, draw=black] (infty) at (0,1) {};
		\node[circle, inner sep = 0pt, minimum size =  1pt, fill = black, draw=black] (ram1) at (0,-0.5) {};
		\node[circle, inner sep = 0pt, minimum size =  1pt, fill = black, draw=black] (ram2) at (0,0.5) {};
		\node[circle, inner sep = 0pt, minimum size =  1pt, fill = black, draw=black] (c1) at (-0.75,0.75) {};
		\node[circle, inner sep = 0pt, minimum size =  1pt, fill = black, draw=black] (c2) at (0.75,-0.75) {};
		\node[left] (mgauss) at (0,0) {$\zeta_{0,1}$};
		\node[right] (m1) at (1,0) {$1$};
		\node[below] (m0) at (0,-1) {$0$};
		\node[above] (minfty) at (0,1) {$\infty$};
		\node[above left] (mc1) at (-0.75,0.75) {$c_2$};
		\node[below right] (mc2) at (0.75,-0.75) {$c_1$};

		\draw (infty) to (0);
		\draw (gauss) to (1);
		\draw[very thick] (c1) to (ram2);
		\draw[very thick] (c2) to (ram1);
		\draw[very thick] (0) to (ram1);
		\draw[very thick] (1) to (gauss);
		\draw[very thick]  (gauss) to (ram2);
	\end{tikzpicture}
	\caption{} \label{fig:mphi=2-1}
\end{minipage} &

\begin{minipage}{0.3\hsize}
     	\begin{tikzpicture}[scale=2.0, auto]
		\node[circle, inner sep = 0pt, minimum size =  1pt, fill = black, draw=black] (gauss) at (0,0) {};
		\node[circle, inner sep = 0pt, minimum size =  1pt, fill = black, draw=black] (0) at (0,-1) {};
		\node[circle, inner sep = 0pt, minimum size =  1pt, fill = black, draw=black] (1) at (1,0) {};
		\node[circle, inner sep = 0pt, minimum size =  1pt, fill = black, draw=black] (infty) at (0,1) {};
		\node[circle, inner sep = 0pt, minimum size =  1pt, fill = black, draw=black] (ram1) at (0,-0.5) {};
		\node[circle, inner sep = 0pt, minimum size =  1pt, fill = black, draw=black] (c2) at (-0.75,0.75) {};
		\node[circle, inner sep = 0pt, minimum size =  1pt, fill = black, draw=black] (c1) at (0.75,-0.75) {};
		\node[left] (mgauss) at (0,0) {$\zeta_{0,1}$};
		\node[right] (m1) at (1,0) {$1$};
		\node[below] (m0) at (0,-1) {$0$};
		\node[above] (minfty) at (0,1) {$\infty$};
		\node[above left] (mc1) at (-0.75,0.75) {$c_2$};
		\node[below right] (mc2) at (0.75,-0.75) {$c_1$};

		\draw (infty) to (0);
		\draw (gauss) to (1);
		\draw[very thick] (c2) to (gauss);
		\draw[very thick] (c1) to (ram1);
		\draw[very thick] (0) to (ram1);
		\draw[very thick] (1) to (gauss);
	\end{tikzpicture}
	\caption{} \label{fig:mphi=2-2}
\end{minipage} &

   \begin{minipage}{0.3\hsize}
 	\begin{tikzpicture}[scale=2.0, auto]
		\node[circle, inner sep = 0pt, minimum size =  1pt, fill = black, draw=black] (gauss) at (0,0) {};
		\node[circle, inner sep = 0pt, minimum size =  1pt, fill = black, draw=black] (0) at (0,-1) {};
		\node[circle, inner sep = 0pt, minimum size =  1pt, fill = black, draw=black] (1) at (1,0) {};
		\node[circle, inner sep = 0pt, minimum size =  1pt, fill = black, draw=black] (infty) at (0,1) {};
		\node[circle, inner sep = 0pt, minimum size =  1pt, fill = black, draw=black] (ram1) at (0,0.5) {};
		\node[circle, inner sep = 0pt, minimum size =  1pt, fill = black, draw=black] (ram2) at (0.5,0) {};
		\node[circle, inner sep = 0pt, minimum size =  1pt, fill = black, draw=black] (c1) at (-0.75,0.75) {};
		\node[circle, inner sep = 0pt, minimum size =  1pt, fill = black, draw=black] (c2) at (0.75,-0.75) {};
		\node[left] (mgauss) at (0,0) {$\zeta_{0,1}$};
		\node[right] (m1) at (1,0) {$1$};
		\node[below] (m0) at (0,-1) {$0$};
		\node[above] (minfty) at (0,1) {$\infty$};
		\node[above left] (mc1) at (-0.75,0.75) {$c_1$};
		\node[below right] (mc2) at (0.75,-0.75) {$c_2$};

		\draw (infty) to (0);
		\draw (gauss) to (1);
		\draw[very thick] (c1) to (ram1);
		\draw[very thick] (c2) to (ram2);
		\draw[very thick] (0) to (ram1);
		\draw[very thick] (1) to (ram2);
	\end{tikzpicture}
	\caption{} \label{fig:mphi=2-3}
\end{minipage} 
\end{tabular}
\end{figure}

\begin{figure}
    \centering
    \begin{tabular}{cc}
\begin{minipage}{0.4\hsize}
	\begin{tikzpicture}[scale=2.0, auto]
		\node[circle, inner sep = 0pt, minimum size =  1pt, fill = black, draw=black] (gauss) at (0,0) {};
		\node[circle, inner sep = 0pt, minimum size =  1pt, fill = black, draw=black] (0) at (0,-1) {};
		\node[circle, inner sep = 0pt, minimum size =  1pt, fill = black, draw=black] (1) at (1,0) {};
		\node[circle, inner sep = 0pt, minimum size =  1pt, fill = black, draw=black] (infty) at (0,1) {};
		\node[circle, inner sep = 0pt, minimum size =  1pt, fill = black, draw=black] (ram1) at (-0.5,0) {};
		\node[circle, inner sep = 0pt, minimum size =  1pt, fill = black, draw=black] (ram2) at (0.5,0.0) {};
		\node[circle, inner sep = 0pt, minimum size =  1pt, fill = black, draw=black] (c1) at (-0.75,0.75) {};
		\node[circle, inner sep = 0pt, minimum size =  1pt, fill = black, draw=black] (c2) at (0.75,-0.75) {};
		\node[circle, inner sep = 0pt, minimum size =  1pt, fill = black, draw=black] (beta) at (-1,-0) {};
		\node[above right] (mgauss) at (0,0) {$\zeta_{0,1}$};
		\node[right] (m1) at (1,0) {$1$};
		\node[below] (m0) at (0,-1) {$0$};
		\node[above] (minfty) at (0,1) {$\infty$};
		\node[above left] (mc1) at (-0.75,0.75) {$c_1$};
		\node[below right] (mc2) at (0.75,-0.75) {$c_2$};
		\node[left] (mbeta) at (-1,0) {$2\beta$};

		\draw (infty) to (0);
		\draw (gauss) to (1);
		\draw (gauss) to (beta);
		\draw[very thick] (c1) to (ram1);
		\draw[very thick] (c2) to (ram2);
		\draw[very thick] (1) to (ram2);
		\draw[very thick] (0) to (gauss);
		\draw[very thick]  (gauss) to (ram1);
	\end{tikzpicture}
	\caption{} \label{fig:mphi=2-4}
	\end{minipage} &

\begin{minipage}{0.4\hsize}
	\begin{tikzpicture}[scale=2.0, auto]
		\node[circle, inner sep = 0pt, minimum size =  1pt, fill = black, draw=black] (gauss) at (0,0) {};
		\node[circle, inner sep = 0pt, minimum size =  1pt, fill = black, draw=black] (0) at (0,-1) {};
		\node[circle, inner sep = 0pt, minimum size =  1pt, fill = black, draw=black] (1) at (1,0) {};
		\node[circle, inner sep = 0pt, minimum size =  1pt, fill = black, draw=black] (infty) at (0,1) {};
		\node[circle, inner sep = 0pt, minimum size =  1pt, fill = black, draw=black] (ram1) at (-0.5,0) {};
		\node[circle, inner sep = 0pt, minimum size =  1pt, fill = black, draw=black] (c1) at (-0.75,0.75) {};
		\node[circle, inner sep = 0pt, minimum size =  1pt, fill = black, draw=black] (c2) at (-0.75,-0.75) {};
		\node[circle, inner sep = 0pt, minimum size =  1pt, fill = black, draw=black] (beta) at (-1,-0) {};
		\node[above right] (mgauss) at (0,0) {$\zeta_{0,1}$};
		\node[right] (m1) at (1,0) {$1$};
		\node[below] (m0) at (0,-1) {$0$};
		\node[above] (minfty) at (0,1) {$\infty$};
		\node[above left] (mc1) at (-0.75,0.75) {$c_1$};
		\node[below left] (mc2) at (-0.75,-0.75) {$c_2$};
		\node[left] (mbeta) at (-1,0) {$\alpha$};

		\draw (infty) to (0);
		\draw (gauss) to (1);
		\draw (gauss) to (beta);
		\draw[very thick] (c1) to (ram1);
		\draw[very thick] (c2) to (ram1);
		\draw[very thick] (1) to (gauss);
		\draw[very thick] (0) to (gauss);
	\end{tikzpicture}
	\caption{} \label{fig:mphi=2-5}
\end{minipage} 
\end{tabular}
\end{figure}
 

\subsection{The case (2-2-2-2-1-1)} \label{deg2-2}
 
The reduction of the Wronskian is
\begin{align*}
    \textstyle{\overline{\mathop{\mathrm{Wr}}}_{\phi}}(z)=z(z-1)^3,
\end{align*}
i.e.\ $\overline{c}_1=\overline{c}_2=1$. 

The Wronskian of $\overline{\phi}$ is
\[
\text{Wr}_{\overline{\phi}}(z)=z(z-1).
\]
The reduction of the $2$ remaining critical points are both $1$, from which we need more detailed analysis in order to determine the ramification locus. Since $\widetilde{\alpha}=\widetilde{\gamma}=1$ and $\widetilde{\beta}=1/2$, we have some $p$, $q\in \{z\in K:|z|<1\}$ such that
\begin{gather*}
    \beta=\frac{1}{2}+p, \text{ and} \\
    \gamma=1+q.
\end{gather*}
Since Wr$_{\phi}(1)=0$, we have that
\[
\alpha=\frac{1-2p+q+2pq}{1-2p+4pq}.
\]
The solution $c_{\pm}$ other than $\psi$ is
\begin{align*}
    c_{\pm}&=\alpha+\beta-\frac{1}{2}\pm\sqrt{\left(\alpha+\beta-\frac{1}{2}\right)^2-2\alpha\beta\gamma} \\
    &=\frac{1-p+q+2pq-2p^2+4p^2q}{1-2p+4pq}\pm\frac{\sqrt{R}}{1-2p+4pq},
\end{align*}
where we put $R$ to be the terms inside of the root i.e.\ 
\begin{align*}
    R&=(1-2p+4pq)^2\left(\left(\alpha+\beta-\frac{1}{2}\right)^2-2\alpha\beta\gamma\right) \\
    &=p^2-2pq-4p^3+8p^2q-6pq^2+4p^4-4pq^3-16p^4q+24p^3q^2-16p^2q^3+16p^4q^2-16p^3q^3.\\
\end{align*}

Therefore, $1-c_{\pm}$ is
\begin{align*}
    1-c_{\pm}=\frac{p+q-2pq-2p^2+4p^2q}{1-2p+4pq}\mp\frac{\sqrt{R}}{1-2p+4pq}.
\end{align*}

We compare the absolute value of the terms appeared in $1-c_{\pm}$ for each of the following $5$ cases;
\begin{description}
    \item[Case 1] $|p|<|q|$;
    \item[Case 2] $|p|=|q|$ and $|p+q|<|p|$;
    \item[Case 3] $|p|>|q|$;
    \item[Case 4] $|p|=|q|$ and $|p-2q|<|p|$;
    \item[Case 5] $|p|=|q|$ and $|4p+q|<|p|$
    \item[Case 6] $|p|=|q|=|p+q|=|p-2q|=|4p+q|$;
\end{description}

Before analyzing them, let us state a lemma which is used several times in the following arguments.
\begin{lmm} \label{lmm:connectedlocus}\textit{
 In the above notation, the ramification locus is connected if $|1-c_-|=|1-c_+|$.
}\end{lmm}
\begin{proof}
If not, the ramification locus consists of $2$ segments. If one segment connects $0$ and $1$, then it must intersect with the other one at $\zeta_{1,|1-c_{+}|}$. By the same argument, in any other possibilities of the $2$ segments, they must intersect at $\zeta_{1,|1-c_{+}|}$, too. This is contradiction.
\end{proof}

\begin{figure}[b]
    \centering
    \begin{tabular}{ccc}

\begin{minipage}{0.3\hsize}
	\begin{tikzpicture}[scale=2.0, auto]
		\node[circle, inner sep = 0pt, minimum size =  1pt, fill = black, draw=black] (gauss) at (0,0) {};
		\node[circle, inner sep = 0pt, minimum size =  1pt, fill = black, draw=black] (0) at (0,-1) {};
		\node[circle, inner sep = 0pt, minimum size =  1pt, fill = black, draw=black] (1) at (1,0) {};
		\node[circle, inner sep = 0pt, minimum size =  1pt, fill = black, draw=black] (infty) at (0,1) {};
		\node[circle, inner sep = 0pt, minimum size =  1pt, fill = black, draw=black] (ram2) at (0.5,0) {};
		\node[circle, inner sep = 0pt, minimum size =  1pt, fill = black, draw=black] (c1) at (0.75,0.75) {};
		\node[circle, inner sep = 0pt, minimum size =  1pt, fill = black, draw=black] (c2) at (0.75,-0.75) {};
		\node[above right] (mgauss) at (0,0) {$\zeta_{0,1}$};
		\node[right] (m1) at (1,0) {$1$};
		\node[below] (m0) at (0,-1) {$0$};
		\node[above] (minfty) at (0,1) {$\infty$};
		\node[above right] (mc1) at (0.75,0.75) {$c_+$};
		\node[below right] (mc2) at (0.75,-0.75) {$c_-$};

		\draw (infty) to (0);
		\draw (gauss) to (1);
		\draw[very thick] (c1) to (ram2);
		\draw[very thick] (c2) to (ram2);
		\draw[very thick] (1) to (gauss);
		\draw[very thick] (0) to (gauss);
	\end{tikzpicture}

	\caption{} \label{fig:mphi=2-6}
\end{minipage} &

\begin{minipage}{0.3\hsize}
	\begin{tikzpicture}[scale=2.0, auto]
		\node[circle, inner sep = 0pt, minimum size =  1pt, fill = black, draw=black] (gauss) at (0,0) {};
		\node[circle, inner sep = 0pt, minimum size =  1pt, fill = black, draw=black] (0) at (0,-1) {};
		\node[circle, inner sep = 0pt, minimum size =  1pt, fill = black, draw=black] (1) at (1,0) {};
		\node[circle, inner sep = 0pt, minimum size =  1pt, fill = black, draw=black] (infty) at (0,1) {};
		\node[circle, inner sep = 0pt, minimum size =  1pt, fill = black, draw=black] (ram1) at (0.75,0) {};
		\node[circle, inner sep = 0pt, minimum size =  1pt, fill = black, draw=black] (ram2) at (0.4,0) {};
		\node[circle, inner sep = 0pt, minimum size =  1pt, fill = black, draw=black] (c1) at (0.75,0.75) {};
		\node[circle, inner sep = 0pt, minimum size =  1pt, fill = black, draw=black] (c2) at (0.75,-0.75) {};
		\node[above right] (mgauss) at (0,0) {$\zeta_{0,1}$};
		\node[right] (m1) at (1,0) {$1$};
		\node[below] (m0) at (0,-1) {$0$};
		\node[above] (minfty) at (0,1) {$\infty$};
		\node[above right] (mc1) at (0.75,0.75) {$c_+$};
		\node[below right] (mc2) at (0.75,-0.75) {$c_-$};

		\draw (infty) to (0);
		\draw (gauss) to (1);
		\draw[very thick] (c1) to (ram1);
		\draw[very thick] (c2) to (ram2);
		\draw[very thick] (1) to (ram1);
		\draw[very thick] (0) to (gauss);
		\draw[very thick]  (gauss) to (ram2);
	\end{tikzpicture}

	\caption{} \label{fig:mphi=2-7}
\end{minipage} &

\begin{minipage}{0.3\hsize}
	\begin{tikzpicture}[scale=2.0, auto]
		\node[circle, inner sep = 0pt, minimum size =  1pt, fill = black, draw=black] (gauss) at (0,0) {};
		\node[circle, inner sep = 0pt, minimum size =  1pt, fill = black, draw=black] (0) at (0,-1) {};
		\node[circle, inner sep = 0pt, minimum size =  1pt, fill = black, draw=black] (1) at (1,0) {};
		\node[circle, inner sep = 0pt, minimum size =  1pt, fill = black, draw=black] (infty) at (0,1) {};
		\node[circle, inner sep = 0pt, minimum size =  1pt, fill = black, draw=black] (ram1) at (0.75,0) {};
		\node[circle, inner sep = 0pt, minimum size =  1pt, fill = black, draw=black] (ram2) at (0.4,0) {};
		\node[circle, inner sep = 0pt, minimum size =  1pt, fill = black, draw=black] (c1) at (0.75,0.75) {};
		\node[circle, inner sep = 0pt, minimum size =  1pt, fill = black, draw=black] (c2) at (0.75,-0.75) {};
		\node[above right] (mgauss) at (0,0) {$\zeta_{0,1}$};
		\node[right] (m1) at (1,0) {$1$};
		\node[below] (m0) at (0,-1) {$0$};
		\node[above] (minfty) at (0,1) {$\infty$};
		\node[above right] (mc1) at (0.75,0.75) {$c_+$};
		\node[below right] (mc2) at (0.75,-0.75) {$c_-$};

		\draw (infty) to (0);
		\draw (gauss) to (1);
		\draw[very thick] (c1) to (ram1);
		\draw[very thick] (c2) to (ram2);
		\draw[very thick] (1) to (ram1);
		\draw[very thick] (0) to (gauss);
		\draw[very thick]  (gauss) to (ram1);
	\end{tikzpicture}

	\caption{} \label{fig:mphi=2-8}
\end{minipage}

\end{tabular}
\end{figure}

\paragraph{\underline{Case $1$}}

In this case, the result is the following;
\begin{prop}\textit{
 In Case 1, the ramification locus is connected. We have $\overline{c_{\pm}}=1$ and $|1-c_{\pm}|=|q|$. The shape is as shown in Figure \ref{fig:mphi=2-6}. The shape is as shown in Figure \ref{fig:mphi=2-7}.
}\end{prop}

\begin{proof}
By the staritforward calculation of the absolute values, we have
\begin{gather*}
    |\sqrt{R}|=|pq|<|q|,\text{ and} \\
    \left|\frac{p+q-2pq-2p^2+4p^2q}{1-2p+4pq}\right|=|q|.
\end{gather*}
Therefore, both of $c_{+}$ and $c_{-}$ satisfies
\[
|c_{\pm}-1|=|q|<1.
\]

In this case, the ramification locus must be connected by Lemma \ref{lmm:connectedlocus}.
\end{proof}

\paragraph{\underline{Case 2}}

In this case, the result is the following;
\begin{prop}\textit{
 In Case 2, then the ramification locus consists of two connected components; one is the segment connecting $0$ and $c_{-}$ and the other is the one connecting $1$ and $c_{+}$. The points $c_{\pm}$ satisfies $\overline{c_{\pm}}=1$, $|1-c_{-}|=|p|$ and $|1-c_{+}|<|p|$.
}\end{prop}
\begin{proof}
Since
\begin{align*}
    R&=p^2-2pq-4p^3+8p^2q-6pq^2+4p^4-4pq^3-16p^4q+24p^3q^2-16p^2q^3+16p^4q^2-16p^3q^3\\
    &=p^2(1-x)
\end{align*}
where $|x|<1$, we have
\begin{align*}
    \sqrt{R}&=p\sqrt{1-x}\\
    &=p\left(1-\frac{x}{2}+(\text{h.o.t. of }x)\right).
\end{align*}

Therefore, we have
\begin{align*}
    1-c_{+}&=\frac{p+q-2pq-2p^2+4p^2q}{1-2p+4pq}\mp\frac{\sqrt{R}}{1-2p+4pq}\\
    &=\frac{q-2pq-2p^2+4pq+px/2+(\text{h.o.t. of }x)}{1-2p+4pq},
\end{align*}
so $|1-c_{+}|<|p|$. A similar calculation shows that $|1-c_{-}|=|p|$.


Next, to have the shape of the ramification locus, we calculate the multiplicity of $\phi$ at $\zeta_{1,|p|}$. To calculate it, we consider the following rational function $\rho$:
\begin{align*}
    \rho(z)&:=\phi(1+z)-\phi(1)\\
    &=\frac{-(1-2p)^2z^3+(4p+q+4pq-8p^2+4p^2q)z^2}{(1-2p+4pq)z-(q-2pq))(2z+1-2p)}
\end{align*}
By setting $\sigma(z)=z-1$, we have $\rho(z)=\sigma\circ\phi\circ\sigma^{-1}$. Hence $\rho$ is conjugation of $\phi$ by $\sigma$. To calculate the multiplicity of $\phi$ at $\zeta_{1,|p|}$, we need to calculate the multiplicity of $\rho$ at $\zeta_{0,|p|}$. For $|z|\leq1$,
\begin{align*}
    \left|\rho(pz)\right|
    =\left|\frac{-p^2\{(1-2p)^2z^3+(4+q/p+4q-8p+4pq)z^2\}}{(1-2p+4pq)z-p(q/p-2q))(2pz+1-2p)}\right|,
\end{align*}
from which we have $\rho(\zeta_{0,|p|})=\zeta_{0,|p|^2}$. Therefore, $\mul_{\rho}(\zeta_{0,|p|})=\deg\overline{\rho(pz)/p^2}$.
\begin{align*}
    \overline{\rho(pz)/p^2}&=\frac{-\{(1-2\overline{p})^2z^3+(4+\overline{q/p}+4\overline{q}-8\overline{p}+4\overline{pq})z^2\}}{(1-2\overline{p}+4\overline{pq})z-\overline{p}(\overline{q/p}-2\overline{q}))(2\overline{p}z+1-2\overline{p})}\\
    &=-z^2+z,
\end{align*}
 which is of degree $2$. 
 
 Therefore, the ramification locus in this case has $2$ components; one connects $0$ and $c_{-}$ and the other connects $1$ and $c_{+}$.
\end{proof}

\paragraph{\underline{Case 3-Case 6}}
\begin{prop}\textit{
 In Case 3, Case 4, Case 5 and Case 6, the ramification locus is connected. In Case 3, Case 4 and Case 5, we have $|1-c_{\pm}|=|p|$ i.e.\ as shown in Figure \ref{fig:mphi=2-6}, and in Case 6, we have exactly one of $|1-c_{\pm}|$ is smaller than $|p|$ and the other is equal to $|p|$, i.e., as shown in Figure \ref{fig:mphi=2-8}.
}\end{prop}

\begin{proof}
By the straightforward calculation of the absolute values, we have $|1-c_{\pm}|=|p|$ in Case 3 and Case 4, where we have the connected ramification locus by Lemma \ref{lmm:connectedlocus}. Hence we consider Case 5 and Case 6. In these cases,



\begin{align*}
    R&=p^2\left(1-\frac{2q}{p}-4p+8q-\frac{6q^2}{p}+4p^2-\frac{4q^3}{p}-16p^2q+24pq^2-16pq^3+16p^2q^2-16pq^3\right)\\
    &=p^2\left(1-\frac{2q}{p}+x\right),
\end{align*}
where $|x|<1$. Hence we have
\begin{align*}
    \sqrt{R}&=p\sqrt{1-\frac{2q}{p}+x}\\
    &=p\left(\sqrt{1-\frac{2q}{p}}+\frac{x}{2\sqrt{1+2q/p}}+(\text{h.o.t. of }x)\right),
\end{align*}


By straightforward calculation, we have
\begin{align*}
    1-c_{\pm} &=\frac{p+q-2pq-2p^2+4p^2q}{1-2p+4pq}\pm \frac{\sqrt{R}}{1-2p+4pq} \\
    &=\frac{p}{1-2p+4pq}\cdot\left(1+\frac{q}{p}-2q-2p+4pq\pm \sqrt{1-\frac{2q}{p}}+\frac{x}{2\sqrt{1+2q/p}}+(\text{h.o.t. of }x)\right) \\
    &=\frac{p}{1-2p+4pq}\cdot\left(1+\frac{q}{p}\pm \sqrt{1-\frac{2q}{p}}+y\right),
\end{align*}
where $|y|<1$. Therefore, we have $|1-c_{+}|<|p|$ or $|1-c_{-}|<|p|$ happens when
\begin{align*}
    \left|1+\frac{q}{p}\pm\sqrt{1-\frac{2q}{p}}\right|<1.
\end{align*}
This is equivalent to the condition that $1+\overline{2q/p} + (\overline{q/p})^2 = 1-\overline{2q/p}$, whence
\[
\overline{q/p}(4+\overline{q/p})=0.
\]
Since $\overline{q/p}\neq0$ by $|p|=|q|$, This occurs when $|4p+q|<|p|$ i.e. Case 5. In Case 6, we have $|1-c_{\pm}|=|p|$ i.e.\ the ramification locus is connected by Lemma \ref{lmm:connectedlocus}.

In Case 5,
\begin{align*}
    \frac{\rho(qz)}{q^2}=\frac{-q^3(1-2p)^2z^3+q^3(4p/q+1+4p-8p^2/q+4p^2)z^2}{q^3((1-2p+4pq)z-(1-2p))(2qz+1-2p)}.
\end{align*}

Therefore, 
\begin{align*}
    \overline{\rho}(z)&=\frac{-z^3+(4\overline{p/q})z^2}{z-1}\\
    &=\frac{z^3}{z-1}.
\end{align*}

Since $\mul_{\phi}(\zeta_{1,|p|})=\deg \tilde{\rho}=3$, the ramification component is always connected in this case, too.
\end{proof}

\end{document}